\documentclass[a4paper,11pt]{amsart}
\usepackage{amssymb}
\usepackage{amsmath}
\usepackage{mathrsfs}
\usepackage{amsmath,amssymb,amsthm,latexsym,amscd,mathrsfs}
\usepackage{indentfirst}
\usepackage{graphicx}
\usepackage{booktabs}
\usepackage{array}
\usepackage{chemarrow}
\newcommand{\ROM}[1]{\mathrm{\uppercase\expandafter{\romannumeral#1}}}
\theoremstyle{definition}
\newtheorem{theorem}{Theorem}
\newtheorem{lemma}[theorem]{Lemma}

\newtheorem{proposition}[theorem]{Proposition}
\newtheorem{rem}[theorem]{Remark}

\newtheorem{problem}[theorem]{Problem}
\newtheorem{conjecture}[theorem]{Conjecture}

\makeatletter


\title[Clifford algebra, isoparametric foliation and related geometric constructions]{\textbf{Clifford algebra, isoparametric foliation and related geometric constructions}}

\author[C. Qian]{Chao Qian}\address{School of Mathematics and Statistics, Beijing Institute of Technology, Beijing 100081, P.R. China}
\email{6120150035@bit.edu.cn}

\author[Z. Z. Tang]{Zizhou Tang}\address{Chern Institute of Mathematics, Nankai University, Tianjin 300071, P.R.
China}\email{zztang@nankai.edu.cn}
\thanks {The project is partially supported by the NSFC (No. 11871282, No. 11401560, No. 11571339) and Nankai Zhide foundation.}

\thanks{}

\thanks{}

\begin{document}

\maketitle

Dedicated to Professor Chiakuei Peng on the Occasion of His $75$th Birthday

\begin{abstract}
Based on representation theory of Clifford algebra, Ferus, Karcher and M\"{u}nzner constructed a series of isoparametric foliations. In this paper, we will survey recent studies on isoparametric hypersurfaces of OT-FKM type and investigate related geometric constructions with mean curvature flow.
\end{abstract}

\section{Introduction}
Let $N$ be a connected complete Riemannian manifold. A non-constant
smooth function $f$ on $N$ is called \emph{transnormal}, if there
exists a smooth function $b:\mathbb{R}\rightarrow\mathbb{R}$ such
that the gradient of $f$ satisfies $|\nabla f|^2=b(f)$.
Moreover, if there exists another function
$a:\mathbb{R}\rightarrow\mathbb{R}$ so that the Laplacian of $f$ satisfies $\triangle f=a(f)$, then $f$ is said to be
\emph{isoparametric}. Each regular level hypersurface of $f$ is then called an
\emph{isoparametric hypersurface}. It was proved by Wang (see \cite{Wa87})
that each singular level set is also a smooth
submanifold (not necessarily connected), the so-called \emph{focal submanifold}.
The whole family of isoparametric hypersurfaces together with the focal
submanifolds form a singular Riemannian foliation,
which is called the \emph{isoparametric foliation}. For recent study of isoparametric functions
on general Riemannian manifolds, especially on exotic spheres, see \cite{GT13}, \cite{GT14} and \cite{QT15}.

E. Cartan firstly gave a systematic study on isoparametric hypersurfaces in real
space forms and proved that an isoparametric hypersurface is exactly
a hypersurface with constant principal curvatures in these cases.
For the spherical case (the most interesting and complicated case), Cartan obtained
the classification result under the assumption that the number of the distinct principal curvatures is at most $3$.
Later, H. F. M\"{u}nzner \cite{Mu80} extended widely Cartan's work.
Precisely, given an isoparametric hypersurface $M^n$ in $S^{n+1}(1)$, let $\xi$ be
a unit normal vector field along $M^n$ in $S^{n+1}(1)$, $g$ the number
of distinct principal curvatures of $M$,
$$\cot \theta_{\alpha}~\;
(\alpha=1,\cdots,g\; ;\;~ 0<\theta_1<\cdots<\theta_{g} <\pi)$$
 the principal
curvatures with respect to $\xi$ and $m_{\alpha}$ the multiplicity
of $\cot \theta_{\alpha}$. M\"{u}nzner proved that $m_{\alpha}=m_{\alpha+2}$
(indices mod $g$),
$\theta_{\alpha}=\theta_1+\frac{\alpha-1}{g}\pi$ $(\alpha = 1,\cdots,
g)$, and there exists
a homogeneous polynomial $F: \mathbb{R}^{n+2}\rightarrow \mathbb{R}$ of degree $g$,
the so-called\emph{ Cartan-M\"{u}nzner polynomial}, satisfying
\begin{equation}\label{ab}
\left\{ \begin{array}{ll}
|\tilde{\nabla} F|^2= g^2r^{2g-2}, \nonumber\\
~\tilde{\triangle} F~~=\frac{m_2-m_1}{2}g^2r^{g-2},\nonumber
\end{array}\right.
\end{equation}
where $r=|x|$, $m_1$ and $m_2$ are the two multiplicities, and $\tilde{\nabla}, \tilde{\triangle}$ are
Euclidean gradient and Laplacian, respectively. Moreover,
M\"{u}nzner obtained the remarkable result that
$g$ must be $1, 2, 3, 4$ or $6$ (see a new simplified proof by Fang \cite{Fa17}). Since then,
the classification of isoparametric hypersurfaces with $g=4$ or $6$
in a unit sphere has been one of the most challenging problems in
differential geometry.

Recently, due to the classification theorem of Chi (see \cite{CCJ07}, \cite{Im08}, \cite{Ch11}, \cite{Ch13} and \cite{Ch16}), an
isoparametric hypersurface with $g=4$ in a unit sphere must be homogeneous
or of OT-FKM type (see below). For $g=6$,
R. Miyaoka \cite{Mi13}, \cite{Mi16} completed
the classification by showing that isoparametric hypersurfaces
in this case are always homogeneous.

Let us now recall the isoparametric hypersurfaces of
OT-FKM type (c.f. \cite{FKM81}). Given a symmetric Clifford system
$\{P_0,\cdots,P_m\}$ on $\mathbb{R}^{2l}$,
\emph{i.e.}, $P_0, ..., P_m$ are symmetric matrices
satisfying $P_{\alpha}P_{\beta}+P_{\beta}P_{\alpha}=2\delta_{\alpha\beta}I_{2l}$, Ferus, Karcher and
M\"{u}nzner defined a polynomial
$F:\mathbb{R}^{2l}\rightarrow \mathbb{R}$ by
$$ F(x) = |x|^4 - 2\displaystyle\sum_{\alpha = 0}^{m}{\langle
P_{\alpha}x,x\rangle^2}.$$
 They verified that $f=F|_{S^{2l-1}(1)}$ is
an isoparametric function on $S^{2l-1}(1)$ and each level
hypersurface of $f$ has $4$ distinct constant
principal curvatures with $(m_1, m_2)=(m, l-m-1)$, provided
$m>0$ and $l-m-1> 0$, where $l = k\delta(m)$ $(k=1,2,3,\cdots)$
and $\delta(m)$ is the dimension of an irreducible module of
the Clifford algebra $C_{m-1}$. As usual, for OT-FKM type, we
denote the two focal submanifolds by $M_+=f^{-1}(1)$ and
$M_-=f^{-1}(-1)$, which have codimensions $m_1+1$ and $m_2+1$
in $S^{2l-1}(1)$, respectively.

\section{Generalizations of OT-FKM construction}
In this section, we will discuss an interesting construction in \cite{QT16}. Inspired by the OT-FKM construction, for a symmetric Clifford system
$\{P_0,\cdots,P_m\}$ on $\mathbb{R}^{2l}$ with the Euclidean metric $\langle\cdot, \cdot\rangle$, we define for $0\leq i\leq m$
$$ M_i:=\{x\in S^{2l-1}(1)~ |~\langle P_0x, x\rangle=\langle P_1x, x\rangle =\cdots=\langle P_ix, x\rangle=0\},$$
and then we have a sequence
$$M_m=M_+\subset M_{m-1}\subset\cdots\subset M_0
\subset S^{2l-1}(1).$$
For $0 \leq i \leq m-1$, it is natural to define a function $f_i:M_i\rightarrow \mathbb{R}$~~by
~~$f_i(x)=\langle P_{i+1}x, x\rangle$ for $x\in M_i$ (see also \cite{TY12'}).

Similarly, by defining for $1\leq i \leq m$, $$N_i:=\{x\in S^{2l-1}(1)~|~\langle P_0x, x\rangle^2+
\langle P_1x, x\rangle^2+\cdots+\langle P_ix, x\rangle^2=1\},$$
we construct
another sequence (to understand the relation of inclusion, see \cite{QT16})
$$N_{1}\subset N_2\subset \cdots\subset N_m=M_-\subset S^{2l-1}(1).$$
And for $2 \leq i \leq m$, we define a function $g_i:N_i\rightarrow \mathbb{R}$
by $g_i(x)=\langle P_{i}x, x\rangle$ for $x\in N_i.$
\begin{theorem}\label{filtration}(\cite{QT16})
Assume the notations as above.
\item[(1).] For $0 \leq i \leq m-1$, the function $f_i:M_i\rightarrow \mathbb{R}$ with
$\mathrm{Im}(f_i)=[-1,~1]$ is an isoparametric function satisfying $$|\nabla f_i|^2=4(1-f_i^2),~
\triangle f_i~~=-4(l-i-1)f_i.$$
For any $c\in (-1,~1)$, the regular level set $\mathcal{U}_c=f_i^{-1}(c)$ has $3$
distinct principal curvatures $$-\sqrt{\frac{1-c}{1+c}}, \;0,\; \sqrt{\frac{1+c}{1-c}}$$ with
multiplicities $l-i-2$, $i+1$, and $l-i-2$ respectively,
w.r.t. the unit normal $\xi=\frac{\nabla f_i}{|\nabla f_i|}$.
For $c=\pm 1$, the two focal submanifolds $\mathcal{U}_{\pm 1}=f_i^{-1}(\pm 1)$ are both
isometric to $S^{l-1}(1)$ and are totally geodesic in $M_i$.

Particularly, we have a minimal isoparametric sequence
$$ M_m\subset M_{m-1}\subset\cdots\subset M_0\subset S^{2l-1}(1),$$
 i.e., each $M_{i+1}$ is a minimal isoparametric hypersurface in $M_{i}$
for $0 \leq i \leq m-1$. Moreover, $M_{i+j}$ is minimal in $M_i$.

\item[(2).] Similarly, for $2 \leq i \leq m$, the function $g_i:N_i\rightarrow \mathbb{R}$ with
$\mathrm{Im}(g_i)=[-1,~1]$ is an isoparametric function satisfying $$|\nabla g_i|^2=4(1-g_i^2),~
\triangle g_i~~=-4ig_i.$$
For any $c\in (-1,~1)$, the regular level set $\mathcal{V}_c=g_i^{-1}(c)$ has $3$
distinct principal curvatures $$-\sqrt{\frac{1-c}{1+c}},\; 0, \; \sqrt{\frac{1+c}{1-c}}$$ with
multiplicities $i-1$, $l-i$, and $i-1$ respectively,
w.r.t. the unit normal $\eta=\frac{\nabla g_i}{|\nabla g_i|}$.
For $c=\pm 1$, the two focal submanifolds $\mathcal{V}_{\pm 1}=g_i^{-1}(\pm 1)$ are both
isometric to $S^{l-1}(1)$ and are totally geodesic in $N_i$.

In particular, we get another minimal isoparametric sequence $$N_{1}\subset N_2\subset \cdots\subset
N_m\subset S^{2l-1}(1),$$  i.e., each $N_{i-1}$ is a minimal isoparametric hypersurface
in $N_i$ for $2 \leq i \leq m$. Moreover, $N_{i}$ is minimal in $N_{i+j}$.
\end{theorem}
\begin{rem}Using representations of Clifford algebras, M. Radeschi in \cite{Ra14} generalized isoparametric foliations of OT-FKM type and constructed indecomposable singular Riemannian foliations of higher codimension on round spheres, most of which are non-homogeneous.
\end{rem}
Next we turn to eigenvalues of Laplacian. Given an $n$-dimensional closed Riemannian manifold $M^n$, recall that the
Laplace-Beltrami operator acting on smooth functions on $M$ is an elliptic operator and has a
discrete spectrum
$$\{0=\lambda_0(M)<\lambda_1(M)\leq \lambda_2(M)\leq\cdots\leq \lambda_k(M)\leq\cdots, k
\uparrow\infty\},$$
with each eigenvalue counted with its multiplicity.
Following the way in \cite{TY13} and \cite{TXY14}, the construction of isoparametric functions in Theorem \ref{filtration} implies the following result on eigenvalue estimates.
\begin{theorem}\label{eigenvalue}(\cite{QT16})
Let $\{P_0,\cdots,P_m\}$ be a symmetric Clifford system on $\mathbb{R}^{2l}$.
\item[(1).] For the sequence $M_m\subset M_{m-1}\subset\cdots\subset M_0\subset
S^{2l-1}(1)$, the following inequalities hold

a). $\lambda_k(M_{i})\leq \frac{l-i-2}{l-i-3}\lambda_k(M_{i+1})$
provided that $0 \leq i\leq m-1$ and $l-i-3>0$;

b). $\lambda_k(M_{i+1})\leq 2\lambda_k(S^{l-1}(1))$ provided that $0 \leq i\leq m-1$.

\item[(2).] For the sequence $N_{1}\subset N_2\subset\cdots\subset
N_m\subset S^{2l-1}(1)$, the following inequalities hold

a). $\lambda_k(N_{i})\leq \frac{i-1}{i-2}\lambda_k(N_{i-1})$ provided that $3\leq i\leq m$;

b). $\lambda_k(N_{i-1})\leq 2\lambda_k(S^{l-1}(1))$ provided that $2\leq i\leq m$.
\end{theorem}

As an unexpected phenomenon, the relations between the focal maps
of isoparametric foliations constructed in Theorem \ref{filtration} and harmonic maps were found. To be more precise,
let $M$ and $N$ be closed Riemannian manifolds,
and $f$ a smooth map from $M$ to $N$. The energy functional $E(f)$ is defined by
$E(f)=\frac{1}{2}\int_M|df|^2dV_M.$
The map $f$ is called \emph{harmonic} if it is
a critical point of the energy functional $E$. We refer to \cite{EL78} and \cite{EL88}
for the background and development of this topic. For $N=S^n(1)$, a map $\varphi:M\rightarrow
S^n(1)$ is called an \emph{eigenmap} (\cite{Ta01}) if the $\mathbb{R}^{n+1}$-components are
eigenfunctions of the Laplacian of $M$ and all have the same eigenvalue. In particular,
$\varphi$ is a harmonic map. In 1980, Eells and Lemaire (See p. 70 of \cite{EL83}) posed
the following

\begin{problem}
Characterize those compact manifolds $M$ for which there is an eigenmap $\varphi:M\rightarrow
S^n(1)$ with $\dim(M)\geq n$ ?
\end{problem}
In 1993, Eells and Ratto (See p. 132 of \cite{ER93}) emphasized again that it is
quite natural to study the eigenmaps to $S^n(1)$.
Another application of the construction in Theorem \ref{filtration} is the following

\begin{theorem}\label{eigenmap}(\cite{QT16})
Let $\{P_0,\cdots,P_m\}$ be a symmetric Clifford system
on $\mathbb{R}^{2l}$.

(1). For $0\leq i \leq m-1$, both of the focal maps $\phi_{\pm \frac{\pi}{4}}: M_{i+1} \rightarrow
\mathcal{U}_{\pm 1}\cong S^{l-1}(1)$ defined by
$$\phi_{\pm \frac{\pi}{4}}(x)=\frac{1}{\sqrt{2}}(x\pm P_{i+1}x),~x \in M_{i+1},$$
are submersive eigenmaps with the same eigenvalue $2l-i-3$.

(2). For $2\leq i \leq m$, both of the focal maps $\psi_{\pm \frac{\pi}{4}}:N_{i-1}\rightarrow
\mathcal{V}_{\pm 1}\cong S^{l-1}(1)$ defined by
$$\psi_{\pm \frac{\pi}{4}}(x)=\frac{1}{\sqrt{2}}(x\pm P_{i}x),~x \in N_{i},$$
are submersive eigenmaps with the same eigenvalue $l+i-2$.
\end{theorem}

We conclude this section with talking about progress of two conjectures on minimal submanifolds. Let $W^n$ be a closed Riemannian manifold minimally immersed in
$S^{n+p}(1)$. Let $B$ be the second fundamental form
and define an extrinsic quantity
$$\sigma(W)=\max\{~|B(X,X)|^2~|~X\in TM,~|X|=1\}.$$

In 1986, H. Gauchman \cite{Ga86} established a well known rigidity theorem which
states that if $\sigma (W)<1/3$, then the submanifold $W$ must be totally geodesic.
When the dimension $n$ of $W$ is even, the rigidity theorem above is optimal.
As presented in \cite{Ga86}, there exist minimal submanifolds
in unit spheres which are not totally geodesic, with $|B(X,X)|^2\equiv 1/3$ for any
unit tangent vector $X$.  When the dimension $n$ of $W$ is odd and $p>1$, the conclusion still holds
under a weaker assumption $\sigma(W)\leq \frac{1}{3-2/n}$.

In 1991, P. F. Leung \cite{Le91} proved that if $n$ is odd, a closed minimally
immersed submanifold $W^n$ with $\sigma(W)\leq \frac{n}{n-1}$ is totally geodesic
provided that the normal connection is flat. Based on this fact, he proposed the
following
\begin{conjecture}\label{weak}  If $n$ is odd, $W^n$ is minimally immersed in $S^{n+p}(1)$
with $\sigma(W) \leq\frac{n}{n-1}$, then $W$ is homeomorphic to $S^n$.
\end{conjecture}

By investigating the second fundamental form of the Clifford minimal hypersurfaces
in unit spheres, Leung also posed the following stronger

\begin{conjecture}\label{strong}  If $n$ is odd and $W^n$ is minimally immersed in $S^{n+p}(1)$
with $\sigma(W) <\frac{n+1}{n-1}$, then $W$ is homeomorphic to $S^n$.
\end{conjecture}
For minimal submanifolds in unit spheres with flat normal connections,
Conjecture \ref{strong} was proved by T. Hasanis and T. Vlachos \cite{HV01}.
In fact, they showed that the condition $\mathrm{Ric}(W)$$ >\frac{n(n-3)}{n-1}$ is equivalent
to the inequality $\sigma(W)< \frac{n+1}{n-1}$. Thus in
the case that the normal connection is flat, Conjecture \ref{strong} follows
from Theorem B in \cite{HV01}.

Recall that the examples with even dimensions and $\sigma(W)= 1/3$ given
in \cite{Ga86} originated from the Veronese embeddings of the projective
planes $\mathbb{R}P^2$, $\mathbb{C}P^2$, $\mathbb{H}P^2$ and $\mathbb{O}P^2$ in $S^4(1)$, $S^7(1)$, $S^{13}(1)$ and
$S^{25}(1)$, respectively. Observe that those Veronese submanifolds are
just the focal submanifolds of isoparametric hypersurfaces
in unit spheres with $g=3$. Hence, it is quite natural for us to consider
the case with $g=4$.
\begin{theorem}\label{counter example}(\cite{QT16})
Let $M^n$ be an isoparametric hypersurface in $S^{n+1}(1)$ with $g=4$ and
multiplicities $(m_1, m_2)$,  and denote by $M_+$ and $M_-$ the focal
submanifolds of $M^n$ in $S^{n+1}(1)$ with dimension $m_1+2m_2$ and $2m_1+m_2$
respectively. Then $M_{\pm}$ are minimal in $S^{n+1}(1)$ with
$\sigma(M_{\pm})=1$. However, $M_{\pm}$ are not homeomorphic to the spheres.
\end{theorem}
\begin{rem}
If $m_1$ is odd, $M_+\subset S^{n+1}(1)$ in Theorem \ref{counter example} is
a counterexample to Conjecture \ref{weak} and Conjecture \ref{strong}. Similarly,
if $m_2$ is odd, $M_- \subset S^{n+1}(1)$ is also a counterexample to both
of the conjectures.
\end{rem}

\section{Pinkall-Thorbergsson Construction}
In this section, we will recall the construction in \cite{PT89} and find some interesting geometric properties.

Let $\{E_1, E_2,..., E_{m-1}\}$ be a set of orthogonal matrices on $\mathbb{R}^l$ with the Euclidean metric, which satisfy
$E_{\alpha}E_{\beta}+E_{\beta}E_{\alpha}=-2\delta_{\alpha\beta}\mathrm{Id}$ for $1 \leq \alpha,\beta\leq m-1$.
Define $$P_0:=\left(
\begin{array}{cc}
Id & 0 \\
0 & -Id \\
\end{array}\right),\;
P_1:=\left(
\begin{array}{cc}
0 & Id\\
Id & 0 \\
\end{array}\right),\;
 P_{\alpha}:=\left(
\begin{array}{cc}
0 & E_{\alpha-1}\\
-E_{\alpha-1} & 0 \\
\end{array}\right),$$
for $2 \leq \alpha\leq m.$
 Then $\{P_0, P_1,..., P_{m}\}$ is a set of orthogonal matrices on $\mathbb{R}^{2l}$ with the Euclidean metric, which satisfy
$P_{\alpha}P_{\beta}+P_{\beta}P_{\alpha}=2\delta_{\alpha\beta}\mathrm{Id}$, for $0 \leq \alpha,\beta\leq m$, i.e. $\{P_0, P_1,..., P_m\}$ is a symmetric
Clifford system on $\mathbb{R}^{2l}$. For any $0\leq\alpha\leq m$, $P_{\alpha}$ has eigenvalues $\pm 1$ of equal multiplicity $l$. Denote the eigenspaces of $\pm 1$ for $P_{\alpha}$ by $E_{\pm}(P_{\alpha})$.

According to \cite{PT89}, for $0<t\leq \frac{\pi}{4}$, one can define
$$M_+^{t}:=\{z=(x,y)\in \mathbb{R}^l\oplus\mathbb{R}^l=\mathbb{R}^{2l}:$$
$$|x|=\cos{t}, |y|=\sin{t}, \langle x, y\rangle=0, \langle x, E_{\alpha}y\rangle=0~\mathrm{for}~1\leq \alpha\leq m-1\}.$$
Clearly, $M_+^t$ is an embedded submanifold in $S^{2l-1}(1)\subset \mathbb{R}^{2l}$ of dimension $2l-m-2$. Write $a:=\tan{t}$\; and
$b:=\cot{t}$. Moreover, write $Q_0:=\left(
\begin{array}{cc}
aId & 0 \\
0 & -bId \\
\end{array}\right)$,
and $Q_{\alpha}:=P_{\alpha}$,
for $1 \leq \alpha\leq m$. Then
$$M_+^{t}=\{z=(x,y)\in \mathbb{R}^{2l}~|~|x|^2+|y|^2=1, \langle z, Q_{\alpha}z\rangle=0~\mathrm{for}~0\leq \alpha\leq m\}.$$
 For  $Q_{\alpha}, ~0\leq \alpha\leq m$, we have the following lemma, which will be useful later.
\begin{lemma}\label{Orthogonal identity}
For any $z\in M_+^t$ and $\alpha, \beta\in \{1, 2,..., m\}$, the following identities hold:
\begin{eqnarray}
\langle Q_0Q_0z, Q_0z\rangle&=&-2\cot{2t},\nonumber\\
\langle Q_{\alpha}Q_0z, Q_0z\rangle&=&0,\nonumber\\
\langle Q_0Q_{\alpha}z, Q_0z\rangle&=&0,\nonumber\\
\langle Q_0Q_0z, Q_{\alpha}z\rangle&=&0,\nonumber\\
\langle Q_{\alpha}Q_{\beta}z, Q_0z\rangle&=&0,\nonumber\\
\langle Q_{\alpha}Q_0z, Q_{\beta}z\rangle&=&0,\nonumber\\
\langle Q_0Q_{\alpha}z, Q_{\beta}z\rangle&=&-2\delta_{\alpha\beta}\cot{2t}.\nonumber
\end{eqnarray}
\end{lemma}
\subsection*{3.1 Extrinsic geometry}
In this subsection, we investigate the extrinsic geometric properties of $M_+^t$ in $S^{2l-1}(1)$ .
\begin{proposition}\label{Extrinsic gometry}
For any given $z\in M_+^t$, one has the following statements:

(1) The normal space $N_zM_+^t$ of $M_+^t$ in the unit sphere $S^{2l-1}(1)\subset \mathbb{R}^{2l}$ at $z$ is given by
$$N_zM_+^t=\mathrm{Span}\{Q_0z,Q_1z,\cdots,Q_mz\}.$$
 The tangent space $T_zM_+^t$ of $M_+^t$ at $z$ is given by
 $$T_zM_+^t=\{(u, v)\in \mathbb{R}^l\oplus\mathbb{R}^l=\mathbb{R}^{2l}~|~\langle(u, v), z\rangle=0, \langle (u, v), Q_iz\rangle=0~\mathrm{for}~0\leq i\leq m\}.$$
  Morover, $\mathbb{R}^{2l}=T_zM_+^t\oplus N_zM_+^t\oplus \mathrm{Span}\{z\}$.

(2) For $1\leq \alpha\leq m$, the shape operator $A_{\alpha}$ of $M_+^t$ with respect to $Q_{\alpha}z$ has principal curvatures $1$, $0$, $-1$ of multiplicities
$l-m-1$, $m$, and $l-m-1$, respectively. Moreover,
$$T_zM_+^t=E_+(Q_{\alpha}z)\oplus E_0(Q_{\alpha}z)\oplus E_-(Q_{\alpha}z),$$
where $$E_+(Q_{\alpha}z)=E_-(Q_{\alpha})\cap T_zM_+^t,$$  $$E_0(Q_{\alpha}z)=\mathrm{Span}\{Q_{\alpha}Q_{\beta}z~|~0\leq\beta\leq m,~\beta\neq\alpha\},$$ and $$E_-(Q_{\alpha}z)=E_+(Q_{\alpha})\cap T_zM_+^t$$ are principal distributions of $1$, $0$, and $-1$, respectively.

(3) Let $E_+(Q_{0})$ and $E_-(Q_{0})$ be eigenspaces of $Q_{0}$ with eigenvalues $a$ and $-b$, respectively. For the unit normal vector $Q_0z$, the shape operator $A_0$ has principal curvatures $\cot{t}$, $0$, $-\tan{t}$ of multiplicities
$l-m-1$, $m$, and $l-m-1$, respectively. Moreover,
$$ T_zM_+^t=E_+(Q_{0}z)\oplus E_0(Q_{0}z)\oplus E_-(Q_{0}z),$$
where
$$E_+(Q_{0}z)=E_-(Q_{0})\cap T_zM_+^t,$$
$$E_0(Q_{0}z)=\mathrm{Span}\{Q_{0}^{-1}Q_{\alpha}z~|~1\leq\alpha\leq m\},$$ and $$E_-(Q_{0}z)=E_+(Q_{0})\cap T_zM_+^t$$
are principal distributions of $\cot{t}$, $0$, and $-\tan{t}$, respectively.
\end{proposition}
\begin{proof}
(1) It follows directly from the definition of $M_+^t$.

(2) Observe that for any tangent vector $X\in T_zM_+^t$, $A_{\alpha}X=-(P_{\alpha}X)^{\mathrm{T}}$, where $(P_{\alpha}X)^{\mathrm{T}}$ is the tangent projection
of $P_{\alpha}X$. Define
$$D_+=E_-(Q_{\alpha})\cap T_zM_+^t,$$ $$D_0=\mathrm{Span}\{Q_{\alpha}Q_{\beta}z~|~0\leq\beta\leq m,~\beta\neq\alpha\},$$
and $$D_-=E_+(Q_{\alpha})\cap T_zM_+^t.$$
 Clearly, $D_{\pm}$ are subspaces of $T_zM_+^t$. Moreover, it follows from Lemma \ref{Orthogonal identity} that $D_0$ is also a subspace of $T_zM_+^t$. Then for any $X\in D_+$, we have $A_{\alpha}X=X$; For any $X\in D_0, A_{\alpha}X=0$; For any $X\in D_-, A_{\alpha}X=-X$. Since
 $$D_+=E_-(Q_{\alpha})\cap T_zM_+^t$$
 $$ =\{(u, v)\in \mathbb{R}^{2l}~|~P_{\alpha}(u, v)=-(u, v), \langle(u, v), z\rangle=0, \langle(u, v), Q_{\beta}z\rangle=0~\mathrm{for}~0\leq\beta\leq m,~\beta\neq\alpha\},$$
we have $\mathrm{dim} D_+\geq l-m-1$. Similarly, $\mathrm{dim}D_-\geq l-m-1$. Thus $D_+, D_0$ and $D_-$ are mutually orthogonal subspaces of $T_zM_+^t$ and $$\mathrm{dim} D_++\mathrm{dim} D_0+\mathrm{dim} D_-\geq l-m-1+m+l-m-1=2l-m-2.$$ Then (2) follows easily.

(3) As in the proof of (2), for any tangent vector $X\in T_zM_+^t$, $A_{0}X=-(Q_{0}X)^{\mathrm{T}}$. Define $$D_+=E_-(Q_{0})\cap T_zM_+^t,$$
 $$D_0=\mathrm{Span}\{Q_{0}^{-1}Q_{\alpha}z~|~1\leq\alpha\leq m\},$$
and $$D_-=E_+(Q_{0})\cap T_zM_+^t.$$  It is clear that $D_{\pm}$ are subspaces of $T_zM_+^t$. By a direct computation, we have $$\langle Q_{0}^{-1}Q_{\alpha}z, z\rangle=0, \langle Q_{0}^{-1}Q_{\alpha}z, Q_{\beta}z\rangle=0,~\mathrm{for}~1\leq\alpha\leq m, 0\leq\beta\leq m.$$
It follows that $D_0$ is also a subspace of $T_zM_+^t$. Moreover,
$$\langle Q_0^{-1}Q_{\alpha}z, Q_0^{-1}Q_{\beta}z\rangle=\delta_{\alpha\beta}$$
 for $1\leq\alpha, \beta\leq m$.
Then for any $X\in D_+$, we have $A_{0}X=bX$; For any $X\in D_0, A_{0}X=0$; For any $X\in D_-, A_{0}X=-aX$.
Meanwhile, it is easy to verify that $D_+, D_0, D_-$ are mutually orthogonal subspaces of $T_zM_+^t$.
Since $$D_+=E_-(Q_{0})\cap T_zM_+^t=\{(0, v)\in \mathbb{R}^{2l}~|~\langle(0, v), z\rangle=0, \langle(0, v), P_{\alpha}z\rangle=0~\mathrm{for}~1\leq \alpha\leq m\}.$$  we have $\mathrm{dim} D_+\geq l-m-1$. Similarly, $\mathrm{dim}D_-\geq l-m-1$. Thus $D_+, D_0$ and $D_-$ are mutually orthogonal subspaces of $T_zM_+^t$ and $$\mathrm{dim} D_++\mathrm{dim} D_0+\mathrm{dim} D_-\geq l-m-1+m+l-m-1=2l-m-2.$$ With these arguments, we can prove (3) easily.
\end{proof}
\subsection*{3.2 Scalar curvature}
Let $M^n$ be a submanifold in $S^{n+p}(1)$, and $B$ the second fundamental form as before. Around each point $z\in M^n$, we can choose an adapted moving frame $e_i$'s and $e_{\alpha}$'s. Restricted on $M$,
$\{e_i~|~1\leq i\leq n\}$ is a local orthonormal basis of $TM$, and $\{e_{\alpha}~|~n+1\leq \alpha\leq n+p\}$ is a local orthonormal basis of $NM$. Denote $$A_{\alpha}e_i=\sum_{j=1}^{n}h_{ij}^{\alpha}e_j,$$ $$H^{\alpha}=\mathrm{Tr}(A_{\alpha})=\sum_{i=1}^{n}h_{ii}^{\alpha},$$ $$H=\sum_{\alpha}H^{\alpha}e_{\alpha},\;|H|^2=\sum_{\alpha}(H^{\alpha})^2,\;|B|^2=\sum_{i, j, \alpha}(h_{ij}^{\alpha})^2.$$
The following lemma follows from the Gauss lemma.
\begin{lemma}
The scalar curvature $S$ of $M^n$ with the induced metric in $S^{n+p}(1)$ is given by
$$S=n(n-1)+|H|^2-|B|^2.$$
\end{lemma}

\begin{proposition}\label{Scalar curvature}
The scalar curvature $S^t$ of $M_+^t$ with the induced metric in $S^{2l-1}(1)$ is
$S^t=(2l-m-2)(2l-m-3)-2(l-m-1)(l-1)+(l-m-1)(l-m-2)(\tan^2{t}+\cot^2{t})$. In particular,
the scalar curvature $S^t$ of $M_+^t$ is a positive constant for $0<t\leq \frac{\pi}{4}$, and $S^t\geq S^{\frac{\pi}{4}}$.
\end{proposition}
\begin{proof}
According to Proposition \ref{Extrinsic gometry}, $H^{0}=2(l-m-1)\cot{2t}$, $H^{\alpha}=0$ for $1\leq\alpha\leq m$, and $$|B|^2=2m(l-m-1)+(l-m-1)(\tan^2{t}+\cot^2{t}).$$  For $t=\frac{\pi}{4}$, the scalar curvature of $M_+^{\frac{\pi}{4}}$ with the induced metric in $S^{2l-1}(1)$ is given by
$$ S^{\frac{\pi}{4}}=4(l-m-1)(l-m-2)+2m(l-m-2)+m(m+1)>0.$$
Then the result follows from the lemma above.
\end{proof}

\begin{rem}

(1). For any $0<t<\frac{\pi}{4}$, both $|H|^2$ and $|B|^2$ of $M_+^t\subset S^{2l-1}(1)$ are positive constants.

(2). If $l-m-1=1$, then $S^t=S^{\frac{\pi}{4}}$ for any $0<t\leq\frac{\pi}{4}$.
\end{rem}

\subsection*{3.3 Mean curvature flow}
By Proposition \ref{Extrinsic gometry}, for $0<t< \frac{\pi}{4}$, $M_+^t$ is not a minimal submanifold in $S^{2l-1}(1)$. Hence, it is interesting to consider the behavior of $M_+^t$ under mean curvature flow.
\begin{proposition}\label{abs}
We have a mean curvature flow as follows

(1). For the initial value $F(\cdot,0): M_+\rightarrow S^{2l-1}(1)$, $$ F(x, y; 0)=(\sqrt{2} \cos \beta(0)x, \sqrt{2}\sin\beta(0) y)$$
with $0<\beta(0)<\frac{\pi}{4}$, the mean curvature flow of $F(\cdot, 0)$ is given by
$$F: M_+\times (-\infty, T) \rightarrow S^{2l-1}(1), \;F(x, y; t)=(\sqrt{2} \cos \beta(t)x, \sqrt{2}\sin\beta(t) y),$$
where $\cos{2\beta(t)}=\cos{2\beta(0)}e^{4(l-m-1)t}$ and $1=\cos{2\beta(0)}e^{4(l-m-1)T}$.

(2). The mean curvature flow $F(M_+, t)$ has type I singularity at $T$. More precisely, there exists a constant $C>0$ such that $$\mathrm{sup}_{F(M_+,t)}|B|^2\leq \frac{C}{T-t}, \forall t\in [0, T).$$

(3). As $t\rightarrow T$, $F(M_+, t)$ converges to $S^{l-1}(1)=\{(x,0)\in \mathbb{R}^{2l}~|~|x|=1\}$.
\end{proposition}
\begin{proof}
(1). Consider the map
\begin{eqnarray}
F: M_+\times (-\infty,T)\longrightarrow S^{2l-1}(1)\nonumber
\end{eqnarray}
by $F(x,y;t)=(\sqrt{2} \cos \beta(t)x, \sqrt{2}\sin\beta(t) y)$, where the function $\beta(t)$
is given by $$\cos{2\beta(t)}=\cos{2\beta(0)}e^{4(l-m-1)t}, \;0< \beta(t)<\pi/4.$$ It is clear that the image of $F$
is included in $M_+^{\beta(t)}$. By a direct computation,
\begin{eqnarray}
\frac{\partial F}{\partial t}=-\sqrt{2}\dot{\beta(t)}(\sin\beta(t) x, -\cos \beta(t) y).\nonumber
\end{eqnarray}
On the other hand, by Proposition \ref{Extrinsic gometry}, the mean curvature vector of $M_+^{\beta(t)}$ at the point
$(\sqrt{2} \cos \beta(t)x, \sqrt{2}\sin\beta(t) y)$ is equal to
\begin{eqnarray}
H|_{F(x,y;t)}&=& (l-m-1)(\cot \beta(t)-\tan \beta(t)) Q_0 F(x,y;t) \nonumber\\
&=& 2\sqrt{2}(l-m-1)\cot 2\beta(t)(\sin\beta(t) x, -\cos \beta(t) y).\nonumber
\end{eqnarray}
At last, the definition $\beta(t)$ yields the equality
$$\dot{\beta(t)}=-2(l-m-1)\cot 2\beta(t),$$
and hence the equation of the mean curvature flow
$\frac{\partial F}{\partial t}=H|_{F(x,y;t)}.$

(2). According to Proposition \ref{Extrinsic gometry}, $$\mathrm{sup}_{F(M_+,t)}|B|^2=2m(l-m-1)+(l-m-1)(\tan^2{\beta(t)}+\cot^2{\beta(t)}).$$
By (1),
\begin{eqnarray}
\mathrm{sup}_{F(M_+,t)}|B|^2&=&2m(l-m-1)+(l-m-1)(\tan^2{\beta(t)}+\cot^2{\beta(t)}) \nonumber\\
&=&2m(l-m-1)+(l-m-1)(\frac{1-\cos{2\beta(t)}}{1+\cos{2\beta(t)}}+\frac{1+\cos{2\beta(t)}}{1-\cos{2\beta(t)}}). \nonumber
\end{eqnarray}
Thus
\begin{eqnarray}
\lim_{t \rightarrow T} \mathrm{sup}_{F(M_+, t)}|B|^2(T-t)&=&\lim_{t \rightarrow T}(l-m-1)\frac{e^{4(l-m-1)(T-t)}+1}{e^{4(l-m-1)(T-t)}-1}(T-t)\nonumber\\
&=&\frac{1}{2}.\nonumber
\end{eqnarray}
Then (2) follows easily.

(3). For any $(x, y)\in M_+$, $\lim_{t\rightarrow T} F(x, y; t)=(\sqrt{2}x, 0)$.
\end{proof}
\begin{rem}(1). Roughly speaking, the family $M_+^t$ in $S^{2l-1}(1)$ constitutes a mean curvature flow.

(2). From Proposition \ref{abs} (1), $F(M_+, t)$ converges to $M_+$ as $t\rightarrow -\infty$.
\end{rem}

\end{document}